\documentclass[12pt,final]{amsart}
\usepackage{amsmath,amssymb}

\numberwithin{equation}{section}

\usepackage{color}
\usepackage{soul,graphicx}

\usepackage[color]{showkeys}
\definecolor{refkey}{rgb}{0,0,1}
\definecolor{labelkey}{rgb}{1,0,0}

\newcommand{\ol}{\overline}

\newcommand{\calS}{\mathcal{S}}

\newcommand{\calP}{\mathcal{P}}

\newcommand{\calR}{\mathcal{R}}

\newcommand{\al}{\alpha}
\newcommand{\bt}{\beta}

\newcommand{\dt}{\delta}

\newcommand{\gm}{\gamma}

\newcommand{\eq} [1] {\begin{equation}\label{#1}\quad}
\newcommand{\en} {\end{equation}}

\newcommand{\diag}{\mathop{\rm diag}}

\newcommand{\Cof}{\mathop{\rm Cof}}
\newcommand{\cof}{\mathop{\rm cof}}
\newcommand{\const}{\mathop{\rm Const}}

\newcommand{\tm}{\times}

\newcommand{\sbs}{\subset}

\newcommand{\wdt}{\widetilde}

\newcommand{\iy}{\infty}

\newcommand{\bZ}{\mathbb{Z}}
\newcommand{\bT}{\mathbb{T}}
\newcommand{\bC}{\mathbb{C}}
\newcommand{\bQ}{\mathbb{Q}}
\newcommand{\bD}{\mathbb{D}}

\newtheorem{theorem}{\bf  Theorem}[section]
\newtheorem{lemma}{\bf  Lemma}[section]

\newtheorem{corollary}{\bf \sc Corollary}[section]

\newtheorem{remark}{ \sc Remark}[section]


\begin{document}
	\begin{center}
		{\bf On the exact spectral factorization of rational matrix functions with applications to paraunitary filter banks}\\[5mm]
		Lasha Ephremidze, Gennady Mishuris, Ilya  Spitkovsky
	\end{center}

\vskip+0.6cm

	{\small{\bf Abstract.}  In this paper, we enhance a recent algorithm for approximate spectral factorization of matrix functions, extending its capabilities to precisely factorize rational matrices when an exact lower-upper triangular factorization is available. This novel approach leverages a fundamental component of the improved algorithm for the precise design of rational paraunitary filter banks, allowing for the predetermined placement of zeros and poles. The introduced algorithm not only advances the state-of-the-art in spectral factorization but also opens new avenues for the tailored design of paraunitary filters with specific spectral properties, offering significant potential for applications in signal processing and beyond. }

\vskip+0.2cm

\vskip+0.2cm \noindent  {\small {\em  MSC:} 47A68, 65T60, 15A83}	

\vskip+0.2cm\noindent  {\small {\em Keywords: } Matrix spectral factorization,
	paraunitary matrix functions, matrix completion problem}

\section{Introduction} \label{Sect1}
Spectral factorization is the process by which a positive (scalar or matrix-valued) function $S$ is expressed in the form
\begin{equation}\label{SF1}
S(t)=S_+(t)S_+^*(t), \;\;\;t\in \mathbb{T},
\end{equation}
where $S_+$ can be analytically extended inside the unit circle $ \mathbb{T}$ and $S_+^*$ is its Hermitian conjugate. There are multiple contexts in which this factorization naturally arises, e.g., linear prediction theory of stationary processes \cite{Kolm41, Wie58}  optimal control \cite{AndesonMoore, Davis} digital communications \cite{Fischer, BLM} etc. Spectral factorization is used to construct certain wavelets \cite{Dau} and multiwavelets \cite{SA4} as well. Therefore, many authors contributed to the development  of different computational methods for spectral factorization (see the survey papers  \cite{Kuc, SayKai} and  references therein, and also \cite{Bott13, Jaf} for more recent results). As opposed to the scalar case, in which an explicit formula exists for factorization:
$S_+(z)=\exp\left(\frac 1{4\pi}
\int\nolimits_\bT\frac{t+z}{t-z}\log
S(t)\,dt\right)$, in general, there is no explicit expression for spectral factorization in the matrix case. The existing algorithms for approximate factorization are, therefore, more demanding in the matrix case.

The Janashia-Lagvilava algorithm \cite{JL99, IEEE2011} is a relatively new method of matrix spectral factorization which proved to be effective, see, e.g., \cite{IEEE2018, CNR, TR22}.  Several generalizations of this method can be found in \cite{IEC21, JMAA22}. In particular, the method is capable to factorize some singular matrix functions with a much higher accuracy than other existing methods (see \S 3 in \cite{TR22}).  Nevertheless, the algorithm, as it was designed originally, is not able to factorize exactly even simple polynomial matrices. In the present paper, we cast a new light on the capabilities of the method eliminating the above-mentioned flaw.	
{ The exact matrix spectral factorization is important as it may be used as a key step in the construction of certain wavelet or multiwavelet filter banks with high precision coefficients \cite{SA4}. Furthermore, we construct a wide class of rational paraunitary matrices, including singular ones with some entries having zeros on the boundary; this construction process is of independent interest.}

Let $\calP^+_N$ be the set of polynomials of degree at most $N$ and for $p(z)=\sum_{k=0}^Nc_kz^k\in\calP^+_N$, let $\wdt{p}(z)=\sum_{k=0}^N\ol{c_k}z^{-k}$. Let also $\calP^-_N:=\{\wdt{p}:p\in\calP^+_N\}$. The core of the Janashia-Lagvilava method is a constructive proof of the following

{\bf Theorem} ( \cite[Th. 1]{IEEE2011}, \cite[Th. 1]{EL2014}). {\em Let $F$ be a $($Laurent$)$ polynomial $m\tm m$ matrix of the form
\begin{equation}
\label{2.1}
F(z)=\begin{pmatrix}1&0&0&\cdots&0&0\\
0&1&0&\cdots&0&0\\
0&0&1&\cdots&0&0\\
\vdots&\vdots&\vdots&\vdots&\vdots&\vdots\\
0&0&0&\cdots&1&0\\
\phi_{1}(z)&\phi_{2}(z)&\phi_{3}(z)&\cdots&\phi_{m-1}(z)&1
\end{pmatrix},
\end{equation}
where
$$
\phi_j\in\mathcal{P}^-_N, \;\;j=1,2,\ldots, m-1.
$$
Then, there exists a unique paraunitary matrix polynomial of the form
\begin{equation}
\label{2.5}
U(z)=\begin{pmatrix}u_{11}(z)&u_{12}(z)&\cdots&u_{1m}(z)\\
u_{21}(z)&u_{22}(z)&\cdots&u_{2m}(z)\\
\vdots&\vdots&\vdots&\vdots\\
u_{m-1,1}(z)&u_{m-1,2}(z)&\cdots&u_{m-1,m}(z)\\[3mm]
\widetilde{u_{m1}}(z)&\widetilde{u_{m2}}(z)&\cdots&\widetilde{u_{mm}}(z)\\
\end{pmatrix},
\end{equation}
where $\;\;\;u_{ij}(z)\in\mathcal{P}^+_N$, $1\leq i,j\leq m$, with determinant $1$,
\begin{equation} \label{2.4}
 \det U(z)=1, \text{ for all } z \text{ wherever } U(z) \text{ is defined},
\end{equation}
satisfying
\begin{equation} \label{U1I}
U(1)=I_m
\end{equation}}
and such that
\begin{equation} \label{2.3}
FU\in (\calP^+_N)^{m\tm m}.
\end{equation}
Here \eqref{2.3} means that $FU$ is an $m\tm m$ matrix with the entries from $\calP^+_N$. A matrix polynomial $U$ is called paraunitary if
$$
U(z)\wdt{U}(z)=I_m,
$$
where $\wdt{U}(z)=[\wdt{u_{ji}}]$ for $U(z)=[u_{ij}(z)]$, and $I_m$ is the $m\tm m$ identity matrix.

Let $\calR$ be the set of rational functions in the complex plane, $\calR_+\sbs\calR$ be the set of rational functions with the poles outside the open unit disk $\bD:=\{z\in\bC:|z|<1\}$, and let $\calR_-\sbs\calR$ be the set of rational functions with the poles inside $\bD$. It follows readily from well-known facts (see Section \ref{Sect5}) that, in the above theorem, if $\phi_j\in\calR_-$ in \eqref{2.1}, $j=1,2,\ldots,m-1$, then $u_{ij}\in\calR_+$ in \eqref{2.5}, $1\leq i,j\leq m$. Nevertheless, in the existing form, the Janashia-Lagvilava algorithm would find only a polynomial approximation to the rational entries $u_{ij}$. In this paper, we construct them exactly. In particular, we provide a constructive proof of the following
\begin{theorem}\label{Th1}
Let $F$ be an $m\tm m$ matrix function of the form \eqref{2.1}, where $\phi_j\in\calR_-$. Assume also that the poles of functions $\phi_j$ are known exactly. Then one can explicitly construct the unique paraunitary matrix function $U$ of the form \eqref{2.5}, where
\begin{equation}\label{uij}
 u_{ij}\in\calR_+, \;\text{ for } 1\leq i,j\leq m,
\end{equation}
satisfying \eqref{2.4} and \eqref{U1I}, such that
\begin{equation} \label{FUR}
FU\in (\calR_+)^{m\tm m}.
\end{equation}
\end{theorem}
Here, for $u\in\calR$, it is assumed that $\wdt{u}(z)=\ol{u(1/\ol{z})}$, which coincides with the introduced definition for $u\in\calP^+_N$.

The first step in the Janashia-Lagvilava algorithm is Cholesky-like factorization of \eqref{SF1}:
\begin{equation}\label{SMM}
  S(z)=M(z)M^*(z),
\end{equation}
where $M$ is the lower triangular matrix with the corresponding scalar spectral factors on the diagonal:
\begin{equation}\label{TrM}
  M(z)=\begin{pmatrix}f^+_1(z)&0&\cdots&0&0\\
\xi_{21}(z)&f^+_2(z)&\cdots&0&0\\
\vdots&\vdots&\vdots&\vdots&\vdots\\
\xi_{r-1,1}(z)&\xi_{r-1,2}(z)&\cdots&f^+_{r-1}(z)&0\\
\xi_{r1}(z)&\xi_{r2}(z)&\cdots&\xi_{r,r-1}(z)&f^+_r(z)
\end{pmatrix}.
\end{equation}
If $S$ is a polynomial matrix function, then the entries of $M$ are rational functions.

Relying on Theorem \ref{Th1}, we prove the following
\begin{theorem} \label{Th2}
Let $S$ be an $r\tm r$ polynomial matrix function which is positive definite $($a.e.$)$ on $\bT$, and let \eqref{SMM} be its lower-upper factorization. If the  entries of \eqref{TrM} and the poles of the functions $\xi_{ij}$ inside $\bT$ are known exactly, then the spectral factorization of $S$ can also be found exactly.
\end{theorem}

{
\begin{remark}
  We emphasize that the diagonal entries of \eqref{TrM} may have zeros on $\mathbb{T}$, however, we do not require knowledge  of their exact locations.
\end{remark}

\begin{remark}
Observe that \eqref{TrM}  in its turn yields  the exact factorization of the determinant. As demonstrated in \cite{AAM}, knowing the latter is necessary and sufficient for the more general Wiener-Hopf factorization to be carried out exactly. The algorithm described in \cite{AAM}, while allowing to exactly factorize any polynomial matrix functions with a factorable determinant $($possessing an arbitrary set of partial indices, unstable among others$)$, does not work in  singular cases where the determinant has zeros on $\bT$. As mentioned in the previous remark, our algorithm does not have this restriction.
\end{remark}
}

Polynomial paraunitary matrix functions  play an important role in the theory of wavelet matrices and paraunitary filter banks (see \cite{RW, Vai, EL2014}). They are also known as multidimensional finite impulse response (FIR) lossless filters, and designing such filters with specific characteristics is of  significant practical importance. In particular, construction of a matrix FIR lossless filter with a given first row, known as the wavelet matrix completion problem, has a long history with various  solutions proposed by a number of authors \cite{EH,HZ,LLS,Park,EL2014, ACHA21}.
Theorem \ref{Th1} leads to a solution of the rational paraunitary matrix completion problem for a broad class of given first rows. Namely, we prove the following
\begin{theorem} \label{Th3}
  Let
  \begin{equation}\label{U1z}
    V_1=(v_1, v_2,\ldots,\wdt{v_m}), \text{ where }v_i\in\calR_+ \text{ for } i=1,2,\ldots,m,
  \end{equation}
be such that
\begin{equation}\label{112}
V_1(z)\wdt{V_1}(z)=\sum\nolimits_{i=1}^m v_i(z)\wdt{v_i}(z)=1 \left( \Longleftrightarrow   \sum\nolimits_{i=1}^m |v_i(t)|^2=1 \text{ for each }t\in\bT\right).
\end{equation}
If
\begin{equation}\label{Coronal2}
  \sum\nolimits_{i=1}^m |v_i(z)|>0 \text{ for each } z\in\bD,
\end{equation}
then one can precisely construct a paraunitary matrix $V$ with the first row \eqref{U1z}.
\end{theorem}
Theorem \ref{Th3} enables us to design multidimensional rational lossless filters with preassigned zeros and poles of the entries in the first row.

The paper is organized as follows: after notation (Section 2) and preliminary observations (Section 3), we prove some auxiliary lemmas in Section 4. Proofs of Theorems~\ref{Th1} and \ref{Th2} are given in Section 5 and 6, respectively. The matrix completion problem is solved in Section 7, while the last Section 8 provides some numerical examples of exact spectral factorization and rational paraunitary matrix construction.

\section{Notation and definitions} \label{Sect2}

This section summarizes the notation used in the paper, with some already introduced in the introduction.

Let $\bT:=\{z\in\bC:|z|=1\}$ be the unit circle in the complex plane, $$\bT_+= \bD:=\{z\in\bC:|z|<1\} \text{ and } \bT_-=\{z\in\bC:|z|>1\}\cup\{\infty\}.$$

For a set $\calS$, let $\calS^{m\times n}$ be the set of $m\times n$ matrices with the entries from $\calS$. Accordingly, the term 'matrix function' denotes matrices with entries as functions, or, equivalently, functions that yield matrices as values.  $I_m=\diag(1,1,\ldots,1)\in\bC^{m\times m}$ stands for the $m\tm m$ identity matrix and  $0_{m\times n}$ is the $m\times n$ matrix consisting of zeros. For a matrix (or a matrix function) $M=[M_{ij}]$, $M^T=[M_{ji}]$ denotes its transpose, and $M^*=[\ol{M_{ji}}]$ denotes its hermitian conjugate, while $[M]_{m\tm m}$ stands for its upper-left $m\tm m$ principal submatrix. On the other hand, for $a\in\bC$, we let $a^\star=\ol{1/a}$.

Let $\mathcal{P}$ be the set of Laurent polynomials with the coefficients in $\bC$:
\begin{equation}\label{calP}
  \mathcal{P}:=\left\{\sum\nolimits_{k=k_1}^{k_2} c_kz^k: c_k\in\bC,\;k_1,k_2\in\bZ;\;k_1\leq k_2   \right\}.
\end{equation}
We also consider the following subsets of $\mathcal{P}$: $\mathcal{P}^+$, $\mathcal{P}^-$, $\mathcal{P}^+_N$ and $\mathcal{P}^-_N$, where $N$ is a non-negative integer,  which correspond to the cases $k_1=0$, $k_2=0$, $0=k_1\leq k_2=N$, and $-N=k_1\leq k_2=0$ in \eqref{calP}, respectively. So, $\mathcal{P}^+$ is the set of usual polynomials, and $\mathcal{P}^+_N$ is the set of polynomials of degree less than or equal to $N$.

The set of rational functions $\{f=p/q: p,q\in\calP\}$ is denoted by $\calR$, and $\calR_+$ (resp. $\calR_-$) stands for the rational functions with the poles in $\bT_-\cup\bT$ (resp. in $\bT_+$). We assume that functions from $\calR_-$ vanish at $\infty$ and constant functions belong to $\calR_+$, so that $\calR=\calR_+\oplus\calR_-$, i.e., every $f\in\calR$ can be uniquely decomposed as
\begin{equation}\label{Rdec}
  f=f^-+f^+,
\end{equation}
 where $f^-\in\calR_-$ and $f^+\in\calR_+$.

For $f\in\calR$, it is assumed that $\wdt{f}(z)=\ol{f(1/\ol{z})}=\ol{f(z^\star)}$, and for $F=[F_{ij}]\in(\calR)^{m\tm n}$ it is assumed that
$\wdt{F}=[\wdt{F_{ji}}]\in(\calR)^{n\tm m}$. Note that $\wdt{f}(z)=\ol{f(z)}$ and $\wdt{F}(z)=F^*(z)$ for $z\in\bT$.

Of course,
\begin{equation}\label{Rel_Con}
  f,\wdt{f}\in\calR_+  \text{ and } f \text{ is free of poles on }\bT  \Longrightarrow f=\const.
\end{equation}

A matrix function $U\in \calR^{m\tm m}$ is called paraunitary if
$$
U(z)\wdt{U}(z)=I_m
$$
(when we write an equation involving rational functions, we assume it holds wherever the rational functions are defined, which is everywhere except for their poles). Note that $U\in\calR^{m\tm m}$ is paraunitary if and only if $U(z)$ is unitary (i.e., $U(z)U^*(z)=I_m$) for each $z\in\bT$.

A matrix function $S\in \calR^{m\tm m}$ is called positive definite, if $S(z)\in\bC^{m\tm m}$ is positive definite for each $z\in\bT$, except for some isolated points (where the determinant of $S$ might be equal to zero, or some entries of $S$ might have a pole).

If $f$ is an analytic function in a neighborhood of $a\in\bC$, then the $k$-th coefficient of its Taylor series expansion is denoted by $c_k^+\{f,a\}$.

The notation $\langle \cdot,\cdot\rangle_{\bC^m}$ and $\|\cdot\|_{\bC^m}$ are for the standard scalar product and norm on $\bC^m$.
The symbol $\dt_{ij}$ stands for the Kronecker delta, i.e., $\dt_{ij}=1$ if $i=j$ and $\dt_{ij}=0$ otherwise, and $e_j=(\dt_{1j},\dt_{2j},\ldots,\dt_{mj})^T $ is a standard basis vector of $\bC^{m\tm 1} $.

Throughout this paper, we understand the term 'construct' as 'finding the object exactly,' with its precise meaning emerging from the surrounding context. For instance, 'constructing' $p\in \calP$ amounts to finding its coefficients exactly, while 'constructing' $f\in\calR$ means determining coprime $p$ and $q$ such that $f=p/q$. In theory, we can also find exact solutions for certain problems, such as determining Laurent series coefficients or solving linear equations with known matrices and vectors. These facts will be employed implicitly in the subsequent sections.

\section{Preliminary observations } \label{Sect3}

\subsection*{3.1}
We will need  the following simple observation. Let
\begin{equation}\label{nSC}
  \begin{cases}
    a_{11}z_1+b_{11}\ol{z_1}+a_{12}z_2+b_{12}\ol{z_2}+\ldots +a_{1n}z_n+b_{1n}\ol{z_n}=c_1\\
     a_{21}z_1+b_{21}\ol{z_1}+a_{22}z_2+b_{22}\ol{z_2}+\ldots +a_{2n}z_n+b_{2n}\ol{z_n}=c_2\\
     \vdots\\
      a_{n1}z_1+b_{n1}\ol{z_1}+a_{n2}z_2+b_{n2}\ol{z_2}+\ldots +a_{nn}z_n+b_{nn}\ol{z_n}=c_n
  \end{cases}
\end{equation}
be the system of $n$ equations with unknowns $z_1, z_2,\ldots,z_n$. It is equivalent to the following $2n\tm 2n$ system of equations with the unknowns $x_i=\Re(z_i)$ and $y_i=\Im(z_i)$, $i=1,2,\ldots,n$,
\begin{equation}\label{2nSR}
  \begin{cases}
    (a^r_{11}+b^r_{11})x_1+ (b^i_{11}-a^i_{11})y_1+\ldots+ (a^r_{1n}+b^r_{1n})x_n+ (b^i_{1n}-a^i_{1n})y_n=c^r_1\\
    \vdots\\
    (a^r_{n1}+b^r_{n1})x_1+ (b^i_{n1}-a^i_{n1})y_1+\ldots+ (a^r_{nn}+b^r_{nn})x_n+ (b^i_{nn}-a^i_{nn})y_n=c^r_n\\
    (a^i_{11}+b^i_{11})x_1+ (a^r_{11}-b^r_{11})y_1+\ldots+ (a^i_{1n}+b^i_{1n})x_n+ (a^r_{1n}-b^r_{1n})y_n=c^i_1\\
    \vdots\\
    (a^i_{n1}+b^i_{n1})x_1+ (a^r_{n1}-b^r_{n1})y_1+\ldots+ (a^i_{nn}+b^i_{nn})x_n+ (a^r_{nn}-b^r_{nn})y_n=c^i_n
  \end{cases},
\end{equation}
where $a^r=\Re(a)$, $a^i=\Im(a)$, and the same for $b$ and $c$.
\begin{remark}\label{SC}
  If system \eqref{nSC}  has a unique solution, then system \eqref{2nSR} has a unique solution as well and, therefore, the determinant of its $2n\tm 2n$ coefficients matrix  is nonzero.
\end{remark}

\subsection*{3.2}
If a matrix polynomial $S\in\calP_{\{-N,N\}}^{m\tm m}$ is positive definite, then spectral factorization \eqref{SF1} has the form
\begin{equation*}
  S(z)=S_+(z)\wdt{S_+}(z),\;\;z\in\bC\backslash\{0\},
\end{equation*}
where $S_+\in \calP_{\{0,N\}}^{m\tm m}$.  Under the usual requirement that the spectral factor $S_+$ is nonsingular on $\bT_+$, the spectral factor is unique up to a constant right unitary multiple. This theorem is known as the Polynomial Matrix Spectral Factorization Theorem, and its elementary proof is available in   \cite{Eph14}.

Since every positive definite matrix function $R\in\calR^{m\tm m}$ can be represented as a ratio $R=S/P$, where $S\in\calP^{m\tm m}$ and $P\in \calP$ are positive definite, the spectral factorization theorem (alongside with the uniqueness) can be extended to the rational case:
\begin{theorem}\label{RSFTH}
  If $R\in\calR^{m\tm m}$ is positive definite, then there exists an unique $($up to a constant right unitary multiple$)$ $R_+\in\calR_+^{m\tm m}$ such that $\det R_+(z)\not=0$ for each $z\in\bT_+$ and
\begin{equation}\label{Rsf}
  R(z)=R_+(z)\wdt{R_+}(z)\;\;
\end{equation}
for each $z$ where both sides of \eqref{Rsf} are defined.
\end{theorem}
\begin{remark}\label{RemPO1}
  To be specific, if $R_+$ and $Q_+$ are two spectral factors of $R$, then there exists a unitary matrix $U\in\bC^{m\tm m}$ such that $R_+(z)=Q_+(z)U$.
\end{remark}
\begin{remark}\label{RemPO2}
Factorization \eqref{Rsf} provides also the spectral factorization of the determinant
\begin{equation*}
  \det R(z)=\det R_+(z)\wdt{\det R_+}(z).
\end{equation*}
\end{remark}

\subsection*{3.3}
Knowing the coefficients $f_{1k}$ of the expansion of an analytic function $f$ in a neighborhood of $a\in\bC$,
\begin{equation*}
  f(z)=\sum\nolimits_{k=0}^\iy f_{1k}(z-a)^k,\;\;f_{1k}\in\bC,\;\;k=0,1,\ldots,
\end{equation*}
we can (explicitly) compute the coefficients of the expansion of its $l$-th power, $f^l$:
\begin{equation}\label{fl}
  \left[f(z) \right]^l=\sum\nolimits_{k=0}^\iy f_{lk}(z-a)^k,\;\;f_{lk}\in\bC, \;\;k=0,1,\ldots,
\end{equation}
in the same neighborhood for each $l\geq 1$ by the following recursive formula
\begin{equation}\label{flr}
  f_{l+1,k}=\sum\nolimits_{j=0}^{k}f_{l,k-j}f_{1j}\,.
\end{equation}
For the sake of notational convenience, we also assume that $f_{0k}=\dt_{0k}$ for $k=0,1,2,\ldots$.

If $b\in\bT_+$ and a function $\wdt{u}\in\calR_-$ has the form
$$
\wdt{u}(z)=\sum\nolimits_{l=0}^N\frac{c_l}{(z-b)^l}\,,
$$
then
\begin{equation}\label{vdt}
  u(z)=\sum\nolimits_{l=0}^{N}\frac{\ol{c_{l}}}{(1/z-\ol{b})^l}=\sum\nolimits_{l=0}^{N}\frac{\ol{c_{l}}z^l}{(1-\ol{b}z)^l}\,.
\end{equation}

If now $a\in \bT_+$ ($a=b$ is not excluded), we have the expansion
$$
\frac{1}{1-\ol{b}z}=\sum\nolimits_{k=0}^\iy \frac{b^\star}{(b^\star-a)^{k+1}}(z-a)^k
$$
in a neighborhood of $a$ (to be specific, for $|z-a|<|b^\star-a|$) and hence (using $z=z-a+a$)
$$
f(z):=\frac{z}{1-\ol{b}z}=\frac{ab^\star}{b^\star-a}+\sum_{k=1}^\iy \frac{b^\star+b^\star a(b^\star-a)}{(b^\star-a)^{k+1}}(z-a)^k=:\sum_{k=0}^\iy f_{1k}(z-a)^k.
$$
Applying now formulas \eqref{fl}, \eqref{flr}, we can expand \eqref{vdt} in the same neighborhood of $a$ as
$$
u(z)=\sum_{l=0}^{N}\ol{c_{l}}\left(\frac{z}{1-\ol{b}z}\right)^l
=:\sum_{l=0}^{N}\ol{c_{l}}\sum_{k=0}^\iy f_{lk}^{ab}(z-a)^k
=\sum_{k=0}^\iy\left(\sum_{l=0}^{N}f_{lk}^{ab}\ol{c_{l}}\right)(z-a)^k,
$$
where the coefficients $f_{lk}^{ab}$, which depend on $a$ and $b$, can be recursively computed for each $l=1,2,\ldots,N$ and $k=0,1,\ldots$

Thus, in order to compute the first $L$ coefficients $c^+_0, c^+_1,\ldots,c^+_{L-1}$ of the expansion of the function \eqref{vdt} in the neighborhood of $a$, one can use the linear transformation
\begin{equation}\label{cAc}
  (c^+_0, c^+_1,\ldots,c^+_{L-1})^T=A^{ab}_{LN}(\ol{c_1},\ol{c_2},\ldots,\ol{c_N})^T,
\end{equation}
where $A^{ab}_{LN}$ is an $L\tm  N$ matrix whose $kl$-th entry is equal to
$$
[A^{ab}_{LN}]_{kl}=f_{lk}^{ab},\;\; 0\leq k< L, \;1\leq l\leq N.
$$
We emphasize that the entries of $A^{ab}_{LN}$ depend only on $a$, $b$, $L$, and $N$.

\section{Some auxiliary lemmas}\label{Sect4}
For given rational functions $\phi_j\in\calR_-$, $j=1,2,\ldots,m-1$, consider the following system of conditions (cf. \cite[eq. (31)]{EL2014})
\begin{equation} \label{7.1}
\begin{cases}
\phi_1x_m-\widetilde{x_1}\in \mathcal{R}_+,\\
\phi_2x_m-\widetilde{x_2}\in \mathcal{R}_+,\\
\vdots\\
\phi_{m-1}x_m-\widetilde{x_{m-1}}\in \mathcal{R}_+,\\
\phi_1x_1+\phi_2x_2+\ldots+\phi_{m-1}x_{m-1}
+\widetilde{x_m}\in \mathcal{R}_+.
\end{cases}
\end{equation}
We say that a vector function $X=(x_1, x_2,\ldots,x_m)^T\in(\calR_+)^{m\tm 1}$ is a solution of \eqref{7.1} if its coordinate functions {\em are free of the poles on $\bT$} and satisfy the conditions in \eqref{7.1}.

\begin{lemma}\label{L1}
	$($cf. \cite[Lemma 3]{EL2014}$)$ If $X$ and $Y$ are two solutions of \eqref{7.1} $($not necessarily different$)$, then
\begin{equation}\label{7.2}
\sum_{k=1}^{m-1}x_k\widetilde{y_k}+\widetilde{x_m}y_m=\const.
\end{equation}	
\end{lemma}
\begin{proof}
	Since both $X$ and $Y$ are solutions, we have in particular:
$$
\begin{cases} \phi_1y_m-\widetilde{y_1}\in \mathcal{R}_+,\\
\vdots\\
\phi_{m-1}y_m-\widetilde{y_{m-1}}\in \mathcal{R}_+,\\
\phi_1x_1+\phi_2x_2+\ldots+\phi_{m-1}x_{m-1}
+\widetilde{x_m}\in \mathcal{R}_+.
\end{cases}
$$
Taking the linear combination of these conditions with the weights $-x_1,\ldots,-x_{m-1}$ and $y_m$, respectively:
$$
\sum_{k=1}^{m-1}x_k\widetilde{y_k}+\widetilde{x_m}y_m\in\calR_+.
$$
Since the conditions on $X$ and $Y$ are symmetric, we also have
$$
\sum_{k=1}^{m-1}y_k\widetilde{x_k}+\widetilde{y_m}x_m\in\calR_+,
$$
and since the functions $x_i$ and $y_i$ do not have poles on $\bT$ by definition, the relation \eqref{Rel_Con} imply \eqref{7.2}.
\end{proof}

The set of solutions $\calS_m$ of \eqref{7.1} is not a linear space. In order to make it linear, we need to modify it and consider
$$
\calS_{\wdt{m}}:=  \{(x_1,x_2,\ldots,x_{m-1},\wdt{x_m})^T : (x_1,x_2,\ldots,x_{m-1},{x_m})^T\in \calS_m\}.
$$
Then $ \calS_{\wdt{m}}$ becomes a linear space in the usual sense: $X$, $Y\in \calS_{\wdt{m}}$ $\Rightarrow$ $\al X+\bt Y\in \calS_{\wdt{m}}$ for each $\al, \bt\in\bC$, where $\al (x_1,\ldots,x_{m-1},\wdt{x_m})^T= (\al x_1,\ldots,\al x_{m-1},\al\wdt{x_m})^T$ (not $\ol{\al}\wdt{x_m}$ in the last position). From now on, slightly abusing the notation, we may also call $ \calS_{\wdt{m}}= \calS_{\wdt{m}}(\phi_1,\phi_2,\ldots,\phi_{m-1})$ the space of the solutions of \eqref{7.1} (along with $\calS_m$) and denote  its elements $(x_1,\ldots,x_{m-1},\wdt{x_m})^T$ by $\hat{X}$.

Since $\wdt{x}(z)=\ol{x(z)}$ for each $z\in\bT$, Lemma \ref{L1} implies the following
\begin{corollary}\label{Cor4.1}
	If $\hat{X}$ and $\hat{Y}$ are two solutions of \eqref{7.1}, then $\langle \hat{X}(z),\hat{Y}(z)\rangle_{\bC^m}$ is constant on $\bT$. In particular, $\|\hat{X}(z)\|_{\bC^m}$ is constant on $\bT$ and if $\hat{X}(z)=0$ for some $z\in\bT$, then $\hat{X}\equiv 0$ and $X\equiv 0$.
\end{corollary}

\begin{corollary}\label{Cor4.2}
	Let $\hat{X}_1,\hat{X}_2,\ldots,\hat{X}_m$ be $m$ solutions of the system \eqref{7.1} such that
	\begin{equation}\label{9.1}
	\hat{X}_i(1)=(\dt_{i1},\dt_{i2},\ldots,\dt_{im})^T,\;\;\;i=1,2,\ldots,m,
	\end{equation}
and let
$$
C=(c_1,c_2,\ldots,c_m)^T\in\bC^{m\tm 1}.
$$
Then
$$
\hat{X}_C=\sum\nolimits_{i=1}^m c_i\hat{X}_i
$$
is the unique solution of the system \eqref{7.1} for which $ \hat{X}(1)=C.$
\end{corollary}

\section{Constructive proof of Theorem 1.1} \label{Sect5}

For $F$ defined by \eqref{2.1}, consider the positive definite matrix function
\begin{equation}\label{RF}
R(z)=F(z)\wdt{F}(z).
\end{equation}
Due to spectral factorization theorem for rational matrix functions (see Theorem \ref{RSFTH} and Remark \ref{RemPO1}), there exists the spectral factorization \eqref{Rsf} of \eqref{RF} such that
\begin{equation}\label{R1F1}
R_+(1)=F(1).
\end{equation}
\begin{remark}\label{Th1Rem1}
We emphasize that such factor $R_+(z)$ is unique \rm{(see Remark \ref{RemPO1}).}
\end{remark}

Since $\det R(z)=1=\det F(z)$ and spectral factorization yields the factorization of the determinant as well (see Remark \ref{RemPO2}), we have
$$
\det R_+(z)=\const
$$
and, due to \eqref{R1F1}, this constant is equal to 1. Therefore,
\begin{equation}\label{R=1}
  \det R_+(z)=1.
\end{equation}
Consequently, the matrix function
\begin{equation}\label{UFR}
  U(z)=F^{-1}(z)R_+(z),
\end{equation}
which is paraunitary since
$$
U(z)\wdt{U}(z)=F^{-1}(z)R_+(z)\wdt{R_+}(z)\wdt{F^{-1}}(z)=
F^{-1}(z)F(z)\wdt{F}(z)\wdt{F}^{-1}(z)=I_m,
$$
has the determinant equal to 1,
$$
\det U(z)=1.
$$
Let $U(z)=[U_{ij}(z)]_{i,j=1}^m$, and investigate its further properties.

If we write the inverse of the matrix \eqref{2.1} explicitly as
\begin{equation}\label{F-1}
F^{-1}(z)=\begin{pmatrix}1&0&0&\cdots&0&0\\
0&1&0&\cdots&0&0\\
0&0&1&\cdots&0&0\\
\vdots&\vdots&\vdots&\vdots&\vdots&\vdots\\
0&0&0&\cdots&1&0\\
-\phi_{1}(z)&-\phi_{2}(z)&-\phi_{3}(z)&\cdots&-\phi_{m-1}(z)&1
\end{pmatrix},
\end{equation}
it follows from \eqref{UFR} that
$$
U_{ij}\in\calR_+ \text{ for } 1\leq i<m, 1\leq j\leq m.
$$
Since
$$
\wdt{U}(z)=U^{-1}(z)=\frac{1}{\det U(z)}[\Cof(U(z))]^T=[\Cof(U(z))]^T,
$$
we have
$$
\wdt{U_{mj}}=\cof(U_{mj})\in\calR_+,\;\;\text{ for } 1\leq j\leq m,
$$
i.e., the entries of the last row of $U$ are in $\{u\in\calR: \wdt{u}\in\calR_+\}$ (note that since $U$ is a unitary matrix function on $\bT$, it does not have poles on $\bT$). Consequently, we can modify the notation for the entries of $U$ and assume that it has the form \eqref{2.5}, where \eqref{uij} holds.
Thus the matrix function \eqref{UFR} satisfies the conditions of Theorem \ref{Th1} and, taking into account that
$$
  U(1)=I_m
$$
(see \eqref{R1F1}) and such $U$ is unique (see Remark \ref{Th1Rem1}), we will construct it explicitly.

First observe that the columns of $U$ satisfy the system of conditions \eqref{7.1}. Indeed, since $R_+^{-1}\in\calR_+$ as a spectral factor and $U^{-1}=\wdt{U}$, it follows from \eqref{UFR} that
$$
\wdt{U}F^{-1}=R_+^{-1}\in\calR_+.
$$
Thus, if we write the product $\wdt{U}F^{-1}$ explicitly, taking into account equations \eqref{F-1} and
$$
\wdt{U}=\begin{pmatrix}\wdt{u_{11}}&\wdt{u_{21}}&\ldots&\wdt{u_{m-1,\,1}}&u_{m1}\\[2mm]
	\wdt{u_{12}}&\wdt{u_{22}}&\ldots&\wdt{u_{m-1,\,2}}&u_{m2}\\
	\vdots&\vdots &\vdots &\vdots &\vdots  \\
	\wdt{u_{1m}}&\wdt{u_{2m}}&\ldots&\wdt{u_{m-1,\,m}}&u_{mm}\\
\end{pmatrix},
$$
we obtain the first $m-1$ conditions in \eqref{7.1} for each column of $U$. The last condition directly follows from the relations $FU=R_+\in\calR_+$.

\begin{remark}\label{RemUsol}
We emphasize that the columns of $U$ are solutions of the system \eqref{7.1}.
\end{remark}

Suppose now that the functions $\phi_i\in\calR_-$ are of the form
\begin{equation*}
  \phi_i(z)=\sum_{k=1}^{n_i}\sum_{l=1}^{N_{ik}}\frac{\gm_{ikl}}{(z-a_{ik})^l}\,;\;\;\;i=1,2,\ldots,m-1,
\end{equation*}
i.e., $\phi_i$ has poles at $a_{i1}, a_{i2},\ldots, a_{i,n_i}$, $|a_{ik}|<1$, of orders $N_{i1}, N_{i2},\ldots, N_{i,n_i}$, respectively. Assume also that
$$
\cup_{i=1}^{m-1}\{a_{i1}, a_{i2},\ldots, a_{i,n_i}\}=:\{a_{m1}, a_{m2},\ldots, a_{m,n_m}\} \; \text{ and } \;
  N_{m\nu}:=\max\{N_{ik}:a_{ik}=a_{m\nu}\},
$$
 $1\leq\nu\leq n_m$,  i.e., we combine poles of all functions $\phi_i$, $i=1,2,\ldots,m-1$, and set their maximal order at each pole as the order of the pole.

For each $j=1,2,\ldots,m$, we construct the $j$-th column of $U$. To this end, assume that $j$ is fixed and let $u_i=u_{ij}$, $i=1,2,\ldots,m$.
Since functions $\wdt{u_i}$, $1\leq i\leq m$, have poles only in $\bT_+$ (see \eqref{uij}), it can be observed from the first $m-1$ equations of \eqref{7.1} that functions $u_i$, $1\leq i<m$, have the form
\begin{equation}\label{ui}
  \wdt{u_i}(z)=C_i+\sum_{k=1}^{n_i}\sum_{l=1}^{N_{ik}}\frac{C_{ikl}}{(z-a_{ik})^l}
\end{equation}
and it follows from the last relation of \eqref{7.1} that
\begin{equation}\label{um}
  \wdt{u_m}(z)=C_m+\sum_{k=1}^{n_m}\sum_{l=1}^{N_{mk}}\frac{C_{mkl}}{(z-a_{mk})^l}\,.
\end{equation}
The functions \eqref{ui}, for  $i=1,2,\ldots,m-1$, and \eqref{um} overall contain
\begin{equation}\label{mN}
m_0:=  m+\sum_{i=1}^{m}\sum_{k=1}^{n_i}N_{ik}
\end{equation}
unknown coefficients:
\begin{equation}\label{CiCikl}
  C_i,\; 1\leq i\leq m, \text{ and }\; C_{ikl},\; 1\leq i\leq m,\, 1\leq k\leq n_i,\, 1\leq l\leq N_{ik}.
\end{equation}

We will construct a linear algebraic system of equations (with these coefficients as unknowns) consisting of  the same number of equations.
Indeed, $m$ equations can be obtained from the relation that the $j$-th column of $U(1)$ is equal to $e_j$ (see \eqref{U1I}), namely we have
\begin{equation}\label{ui1}
   {u_i}(1)=\ol{C_i}+\sum_{k=1}^{n_i}\sum_{l=1}^{N_{ik}}\frac{\ol{C_{ikl}}}{(1-\ol{a_{ik}})^l}=\dt_{ij}\,,\text{ for } i=1,2,\ldots,m.
\end{equation}
In addition, for each function $\phi_i$, $i=1,2,\ldots,m-1$, considering $i$-th relation of the system \eqref{7.1} which is satisfied by $u_m$ and $u_i$,
\begin{equation}\label{phiix}
  \phi_i u_m-\widetilde{u_i}\in \mathcal{R}_+,
\end{equation}
and equating $N_{ik}$ negative indexed coefficients to $0$ in the Laurent expansion of the  function in \eqref{phiix} in a neighborhood of $a_{ik}$, where $1\leq k\leq n_i$, we get the following $N_{ik}$ equations written in the matrix form
\begin{equation}\label{GmcC}
\left[\begin{matrix}\gamma_{ik1}&\gamma_{ik2}&\gamma_{ik3}&\cdots&\gamma_{ik,N_{ik}-1}&\gamma_{ikN_{ik}}\\
\gamma_{ik2}&\gamma_{ik3}&\gamma_{ik4}&\cdots&\gamma_{ikN_{ik}}&0\\
\gamma_{ik3}&\gamma_{ik4}&\gamma_{ik5}&\cdots&0&0\\
\cdot&\cdot&\cdot&\cdots&\cdot&\cdot\\
\gamma_{ikN_{ik}}&0&0&\cdots&0&0\end{matrix}\right]\cdot
\left[\begin{matrix}c^+_0\{u_m,a_{ik}\}\\c^+_1\{u_m,a_{ik}\}\\ c^+_2\{u_m,a_{ik}\}\\ \vdots\\c^+_{N_{ik}-1}\{u_m,a_{ik}\}\end{matrix}\right]=
\left[\begin{matrix}C_{ik1}\\C_{ik2}\\ C_{ik3}\\ \vdots\\C_{ikN_{ik}}\end{matrix}\right].
\end{equation}
Using  equation \eqref{um} and the relations described by \eqref{cAc}, we can substitute
$$
(c^+_0\{u_m,a_{ik}\}, \cdots\\c^+_{N_{ik}-1}\{u_m,a_{ik}\})^T=\sum\nolimits_{\tau=1}^{n_m}A^{a_{ik}a_{m\tau}}_{N_{ik}N_{m\tau}}
(\ol{C_{m\tau1}}, \ol{C_{m\tau2}}, \cdots,\ol{C_{m\tau N_{m\tau}}})^T
$$
in \eqref{GmcC}.  The resulting system will contain only \eqref {CiCikl} as the unknowns.  Performing the same procedure for each pole $a_{ik}$ of $\phi_i$, we obtain $\sum\nolimits_{k=1}^{n_i}N_{ik}$ equations for each $i=1,2,\ldots,m-1$, and thus in total
\begin{equation}\label{ssN}
 m_1:= \sum\nolimits_{i=1}^{m-1}\sum\nolimits_{k=1}^{n_i}N_{ik}
\end{equation}
additional equations.

Consider now the solution of the last condition in \eqref{7.1}:
$$
\phi_1u_1+\phi_2u_2+\ldots+\phi_{m-1}u_{m-1}+\widetilde{u_m}\in \mathcal{R}_+\,.
$$
Equating again negative indexed coefficients to $0$ in the Laurent expansion of the above function  in a neighborhood of $a_{mk}$, for each $k=1,2,\ldots n_m$, we get the following $N_{mk}$ equations
\begin{equation}\label{GmcC2}
\sum_{i}
\left[\begin{matrix}\gamma_{i\nu1}&\gamma_{i\nu2}&\gamma_{i\nu3}&\cdots&\gamma_{i\nu N_{mk}-1}&\gamma_{i\nu N_{mk}}\\
\gamma_{i\nu2}&\gamma_{i\nu3}&\gamma_{i\nu4}&\cdots&\gamma_{i\nu N_{mk}}&0\\
\gamma_{i\nu3}&\gamma_{i\nu4}&\gamma_{i\nu5}&\cdots&0&0\\
\cdot&\cdot&\cdot&\cdots&\cdot&\cdot\\
\gamma_{i\nu N_{mk}}&0&0&\cdots&0&0\end{matrix}\right]
\left[\begin{matrix}c^+_0\{u_i,a_{mk}\}\\c^+_1\{u_i,a_{mk}\}\\ c^+_2\{u_i,a_{mk}\}\\ \vdots\\c^+_{N_{mk}-1}\{u_i,a_{mk}\}\end{matrix}\right]=
-\left[\begin{matrix}C_{mk1}\\C_{mk2}\\ C_{mk3}\\ \vdots\\C_{mkN_{mk}}\end{matrix}\right].
\end{equation}
Here, the summation is with respect to those $i$es for which $a_{i\nu}=a_{mk}$ for some $\nu\leq n_i$ and it is assumed that $\gamma_{i\nu l}=0$ if $N_{i\nu}<N_{mk}$ and $N_{i\nu}<l\leq N_{mk}$. Again, we can eliminate  extra unknowns in the above equations by making the substitutions (see
\eqref{ui} and \eqref{cAc})
$$
(c^+_0\{u_i,a_{mk}\},  \cdots\\c^+_{N_{mk}-1}\{u_i,a_{mk}\})^T=\sum\nolimits_{\tau=1}^{n_i}A^{a_{mk}a_{i\tau}}_{N_{mk}N_{i\tau}}
(\ol{C_{i\tau1}}, \ol{C_{i\tau2}}, \cdots,\ol{C_{i\tau N_{i\tau}}})^T.
$$
This way, we get
\begin{equation}\label{ssN2}
 m_2:= \sum\nolimits_{k=1}^{n_m}N_{mk}
\end{equation}
additional equations, and summing up $m$ (the number of equations in \eqref{ui1}) with $m_1$ in \eqref{ssN} and $m_2$ in \eqref{ssN2}, we get $m_0$ in \eqref{mN}.
Consequently, we can construct the system of $m_0$ algebraic  equations  with $m_0$ unknowns \eqref{CiCikl}.
Some of these unknowns enter in this equation with their conjugate like $\ol{C_i}$ or $\ol{C_{ikl}}$, however,
the existence and uniqueness of the solution to this  system is known beforehand
and, taking into account Remark \ref{SC}, these unknowns can be found explicitly by the standard way.

\section{Constructive proof of Theorem 1.2} \label{Sect6}
By applying Theorem \ref{Th1}, the proof of Theorem \ref{Th2} closely follows the approach outlined in previous publications \cite{IEEE2011, IEEE2018}. However, it is crucial to note that the obtained intermediate terms can be computed exactly, which is a central focus of this paper. Particularly, throughout this section, when we refer to ``constructing", we understand ``computing exactly".

Note that the triangular factor $M$ in \eqref{TrM}  can be constructed within the field of rational functions using a process similar to the Cholesky factorization algorithm for standard numerical matrices. The key difference is that, instead of extracting square roots as the standard algorithm requires, one performs scalar spectral factorization. It is assumed that  the poles indicated in the theorem can be precisely determined during the construction of such $M$. Observe also that the determinant of \eqref{TrM} can be analytically extended inside $\bT$ everywhere and
\begin{equation}\label{detM}
  \det M(z)\not=0 \;\;\text{ for each }z\in\bT_+
\end{equation}
since it is required that the diagonal functions $f_m$, $m=1,2,\ldots$, are spectral factors.

The spectral factor $S_+$ can be represented as the product
\begin{equation}\label{S3}
S_+(z)=M(z)\mathbf{U}_2(z)\mathbf{U}_3(z)\ldots\mathbf{U}_r(z),
\end{equation}
where each matrix $\mathbf{U}_m$ is paraunitary and has the following block matrix form
\begin{equation}\label{UB}
\mathbf{U}_m(z)=\begin{pmatrix}U_{m}(z)&0_{m\tm (r-m)}\\0_{(r-m)\tm m}&I_{r-m}\end{pmatrix},\;\;m=2,3,\ldots,r
\end{equation}
Matrices \eqref{UB} are constructed recursively in such a way that $[M_m]_{m\tm m}=:[S]_{m\tm m}^+$ is a spectral factor of $[S]_{m\tm m}$, where
\begin{equation*}\label{75}
M_m=M\mathbf{U}_2\mathbf{U}_3\ldots\mathbf{U}_m.
\end{equation*}
This can be achieved by using Theorem \ref{Th1}. Indeed, let us assume that $\mathbf{U}_2, \mathbf{U}_3, \ldots,\mathbf{U}_{m-1}$ are already constructed so that
\begin{equation}\label{SMMm1}
  [S(z)]_{(m-1)\tm (m-1)}=[M_{m-1}(z)]_{(m-1)\tm (m-1)}[\wdt{M_{m-1}}(z)]_{(m-1)\tm (m-1)}.
\end{equation}
Since matrices \eqref{UB} are paraunitary and they have the special structure, we also have
$$
[S(z)]_{m\tm m}=[M_{m-1}(z)]_{m\tm m}[\wdt{M_{m-1}}(z)]_{m\tm m}.
$$
Furthermore, the matrix $[M_{m-1}(z)]_{m\tm m}$ has the form
\begin{gather}
[M_{m-1}]_{m\tm m}=\left[\begin{matrix}& &  S_{(m-1)\tm (m-1)}^+(t)& & \begin{matrix}0_{(m-1)\tm 1}\end{matrix}\\
\zeta_1 & \zeta_2 & \ldots& \zeta_{m-1}& f_m^+
\end{matrix}\right] =
\left[\begin{matrix} S_{(m-1)\tm (m-1)}^+(t) & \begin{matrix}0_{(m-1)\tm 1}\end{matrix}\\
0_{1\tm(m-1)}& f_m^+ \end{matrix}\right]\notag\\   \times
\left[\begin{matrix}& &  I_{m-1}& & \begin{matrix}0_{(m-1)\tm 1}\end{matrix}\\
\zeta_1/f_m^+ & \zeta_2/f_m^+ & \ldots& \zeta_{m-1}/f_m^+& 1
\end{matrix}\right]=
\left[\begin{matrix} S_{(m-1)\tm (m-1)}^+(t) & \begin{matrix}0_{(m-1)\tm 1}\end{matrix}\\
0_{1\tm(m-1)}& f_m^+ \end{matrix}\right] \notag\\  \label{Mmm3}\times
\left[\begin{matrix}& &  I_{m-1}& & \begin{matrix}0_{(m-1)\tm 1}\end{matrix}\\
\phi^+_1 & \phi^+_2 & \ldots& \phi^+_{m-1}& 1
\end{matrix}\right] \left[\begin{matrix}& &  I_{m-1}& & \begin{matrix}0_{(m-1)\tm 1}\end{matrix}\\
\phi^-_1 & \phi^-_2 & \ldots& \phi^-_{m-1}& 1
\end{matrix}\right],
\end{gather}
where $$\zeta_i/f_m^+=\phi^+_i+\phi^-_i,\;\; i=1,2,\ldots,m-1,$$ is the decomposition of a rational function according to the rule \eqref{Rdec}.
The first two factors in the right hand side of  \eqref{Mmm3} belong to $(\calR_+)^{m\tm m}$ and their determinants are free of zeros in $\bT_+$. Assuming now that $F$ is the last matrix in \eqref{Mmm3} and applying Theorem \ref{Th1}, we can find a paraunitary matrix $U_m=U$ of the form \eqref{2.5}, satisfying \eqref{uij} and \eqref{2.4}, such that \eqref{FUR} holds. Hence,
\begin{equation}\label{Mm1Um}
  [M_{m-1}]_{m\tm m} U_m=[M_m]_{m\tm m}
\end{equation}
is a spectral factor of $[S]_{m\tm m}$, and equation \eqref{SMMm1} remains valid if we change $(m-1)$ to $m$. Note that, although the factors in \eqref{Mm1Um} are merely rational matrices, the product $[M_m]_{m\tm m}=[S]_{m\tm m}^+\in(\calP_+)^{m\tm m}$ due to polynomial spectral factorization theorem (see Section 3.2).

Thus, if we accordingly construct all the matrices $\mathbf{U}_2, \mathbf{U}_3,\ldots\mathbf{U}_r$ in \eqref{S3}, we obtain a spectral factor $S_+$.

\begin{remark}
  In this paper, our focus is on exact factorization. If we cannot precisely handle the factor described in Theorem \ref{Th2} as given by \eqref{TrM}, we can still obtain its entries and corresponding poles in $\mathbb{T}$ with any prescribed accuracy. This allows us to derive an approximation of $M$ by a rational matrix function, say $\hat{M}$. We can then proceed with $\hat{M}$ and apply the algorithm as previously described. This results in an approximate spectral factor $\hat{S}_+$. We anticipate that the convergence $\hat{S}_+\to S_+$ should occur under the condition that $\hat{M}\to M$, which will be the focus of our future work.

In our opinion, the ideas developed in this section can also be used to factorize matrices depending on a parameter.
\end{remark}

\section{Completion of paraunitary matrices} \label{Sect7}

In this section we prove Theorem \ref{Th3}. Since a matrix $V$ is paraunitary if and only if $V^T$ is paraunitary, for notational convenience, we assume the given data is the first column of the matrix and construct its completion. Hence we assume that
\begin{equation}\label{V1Com}
 V_1=(v_1, v_2,\ldots,\wdt{v_m})^T\in\calR^{m\tm 1}, \text{ where }v_i\in\calR_+ \text{ for } i=1,2,\ldots,m,
\end{equation}
is given which satisfies $\wdt{V_1}V_1=1$ and we obtain a paraunitary matrix $V$ with the first column $V_1$.
This completion follows the same idea as demonstrated in \cite[Section 5]{EL2014} for polynomial matrices: first we construct a rational row vector
\begin{equation} \label{phiebi}
(\phi_1,\phi_2\ldots,\phi_{m-1})\in\calR_-^{1\tm(m-1)}
\end{equation}
and then, applying the procedures described in the  proof of Theorem \ref{Th1}, we construct the corresponding completion. To this end, we consider the same system of conditions as \eqref{7.1},
\begin{equation} \label{7.12}
\begin{cases} \phi_1v_m-\widetilde{v_1}\in \mathcal{R}_+,\\
\phi_2v_m-\widetilde{v_2}\in \mathcal{R}_+,\\
\vdots\\
\phi_{m-1}v_m-\widetilde{v_{m-1}}\in \mathcal{R}_+,\\
\phi_1v_1+\phi_2v_2+\ldots+\phi_{m-1}v_{m-1}
+\widetilde{v_m}\in \mathcal{R}_+\,,
\end{cases}
\end{equation}
however, now we treat   $v_i\in\calR_+$ (for $1\leq i\leq m$) as known functions and $\phi_j\in\calR_-$ (for $1\leq j <m$) as unknown functions. A necessary and sufficient condition for the existence of this solution is provided by the following
\begin{lemma} \label{LemNSC}
 For given functions  $v_i\in \calR_+$, $i=1,2,\ldots,m$, satisfying
 \begin{equation}\label{erTi}
 \sum\nolimits_{i=1}^m v_i(z)\wdt{v_i}(z)=1,
 \end{equation}
 there exists solution \eqref{phiebi} of system \eqref{7.12}
 $(\phi_1,\phi_2,\ldots,\phi_{m-1})^T$, $\phi_j\in \calR_-$, $j=1,2,\ldots,m-1$,
 if and only if there exist functions $h_i\in\calR_+$, $1\leq i\leq m$, such that
 \begin{equation}\label{Corona2}
   \sum\nolimits_{i=1}^m h_iv_i=1.
 \end{equation}
\end{lemma}

\begin{proof}
General solutions of the first $m-1$ conditions in \eqref{7.12} have the form
\begin{equation}\label{GS1}
\phi_i=\left[\frac{\wdt{v_i}}{v_m}\right]^-+\psi_i, \text{ where } \psi_i\in \calR_- \text{ and }\psi_iv_m\in\calR_+.
\end{equation}
Substituting these relations into the last condition of \eqref{7.12}, we get
$$
\sum_{i=1}^{m-1}\left(\left[\frac{\wdt{v_i}}{v_m}\right]^-+\psi_i\right)v_i+\wdt{v_m}=
\sum_{i=1}^{m-1}\left(\frac{\wdt{v_i}}{v_m}+\psi_i-\left[\frac{\wdt{v_i}}{v_m}\right]^+\right)v_i+\wdt{v_m}
\in\calR_+
$$
Hence,
$$
\sum_{i=1}^{m-1}\frac{\wdt{v_i}v_i+\psi_iv_m v_i}{v_m}+\frac{v_m\wdt{v_m}}{v_m}=:h_m\in\calR_+
$$
and, by virtue of \eqref{erTi},
$$
1+\sum\nolimits_{i=1}^{m-1}\psi_i v_m v_i=h_m v_m.
$$
Thus, if we introduce the notation $h_i=-\psi_i v_m$, $i=1,2,\ldots,m-1$, where $h_i\in\calR_+$ due to the last condition in \eqref{GS1},  we get \eqref{Corona2} and the first part of the lemma is proved.

Suppose now that \eqref{Corona2} holds and define functions
$$
\phi_i=\left[\frac{\wdt{v_i}-h_i}{v_m}\right]^-, \;\;i=1,2,\ldots,m-1.
$$
Then the first $m-1$ conditions in \eqref{7.12} are satisfied and, for the last condition, we have
\begin{gather*}
\sum_{i=1}^{m-1} \left[\frac{\wdt{v_i}-h_i}{v_m}\right]^- v_i+\wdt{v_m}=
\sum_{i=1}^{m-1} \left(\frac{\wdt{v_i}-h_i}{v_m}\right) v_i+\wdt{v_m}-\sum_{i=1}^{m-1} \left[\frac{\wdt{v_i}-h_i}{v_m}\right]^+ v_i \\
=\frac{\sum\nolimits_{i=1}^{m-1}\wdt{v_i}v_i+\wdt{v_m}v_m-\sum\nolimits_{i=1}^{m-1}h_i v_i}{v_m} -\sum_{i=1}^{m-1} \left[\frac{\wdt{v_i}-h_i}{v_m}\right]^+ v_i \\
=\frac{1-\sum\nolimits_{i=1}^{m-1}h_i v_i}{v_m} -\sum_{i=1}^{m-1} \left[\frac{\wdt{v_i}-h_i}{v_m}\right]^+ v_i=h_m-\sum_{i=1}^{m-1} \left[\frac{\wdt{v_i}-h_i}{v_m}\right]^+ v_i\in\calR_+.
\end{gather*}
Hence the last condition of \eqref{7.12} is also satisfied and the lemma is proved.
\end{proof}

\begin{remark}
  It is well known that condition \eqref{Coronal2} is equivalent to  \eqref{Corona2} satisfied  for some polynomials $h_i\in\calP_+$, $1\leq i\leq m$.  Moreover, these polynomials can be explicitly found  by solving the corresponding system of linear algebraic equations.
\end{remark}

We proceed with the proof of Theorem \ref{Th3} as follows. Suppose
$$
V_1(1)=C=(c_1,c_2,\ldots,c_m)^T\in\bC^{m\tm 1}.
$$
We can complete $C$ to the unitary matrix $W$, i.e., we can construct a unitary matrix $W$ with the first column $C$.

Lemma \ref{LemNSC} guarantees that we can construct a row vector \eqref{phiebi} which is the solution of  system \eqref{7.12}. Use these $\phi_i$, $i=1,2,\ldots,m-1$,  to construct a matrix function $F$ defined by equation \eqref{2.1}, and let $U$ be the corresponding paraunitary matrix determined according to Theorem \ref{Th1}. We have that the columns
$$
\hat{U}_j=(u_{1j},u_{2j},\ldots,u_{m-1,j},\wdt{u_{mj}})^T,\;\;j=1,2,\ldots,m,
$$
of the matrix $U$ are solutions of system \eqref{7.1} (see Remark \ref{RemUsol}) and
$$
\hat{U}_j(1)=(\dt_{1j},\dt_{2j},\ldots,\dt_{mj})^T,\;\;\;j=1,2,\ldots,m,
$$
(see \eqref{U1I}). Hence
\begin{equation}\label{U_C}
  V_C=\sum\nolimits_{j=1}^m c_j\hat{U}_j,
\end{equation}
 is the solution of system \eqref{7.1} which satisfies $V_C(1)=C$. However, such solution of the system \eqref{7.1} is unique (see Corollary \ref{Cor4.2}). Therefore, $V_C$ defined by \eqref{U_C}, which is the first column of $ U(z)\cdot W$, coincides with \eqref{V1Com}. Obviously, the matrix
\begin{equation}\label{UW}
V(z)= U(z)\cdot W
\end{equation}
 is paraunitary, therefore the proof of Theorem \ref{Th3} is completed.

\section{Numerical examples} \label{Sect8}
\subsection*{5.1}
In this section, we provide some examples of exact constructions which rely on the  methods presented in this paper. First we factorize the following matrix
\begin{equation}\label{IEE0}
S(z)=
\begin{pmatrix}2z^{-1}+6+2z&11z^{-1}+22+7z\\
7z^{-1}+22+11z&38z^{-1}+84+38z\end{pmatrix},
\end{equation}
which is positive definite, however, it has a singularity (of order 4) at the isolated point $z=1$. In particular, $\det S(z)=-z^{-2}+2-z^2=(z^{-2}-1)(z^2-1)$. An approximate factorization of this matrix, varying in both speed and accuracy, is given in  \cite{IEEE2011, IEEE2018}. Now, by using Theorem \ref{Th1}, we can factorize \eqref{IEE0} exactly as it admits exact lower-upper factorization \eqref{SMM} with
\begin{equation}\label{Mz2}
M(z)=
\begin{pmatrix}b+az&0\\
\frac{7+22z+11z^2}{a+bz}&\frac{1-z^2}{b+az}\end{pmatrix},
\end{equation}
where $a=\sqrt{3-\sqrt{5}}$ and $b=\sqrt{3+\sqrt{5}}$. Arguing as in \eqref{Mmm3}, the matrix \eqref{Mz2} can be represented as
$$
M(z)=
\begin{pmatrix}b+az&0\\
0&\frac{1-z^2}{b+az}\end{pmatrix}
\begin{pmatrix}1&0\\
\phi &1\end{pmatrix},
$$
where $\phi=(7+22z+11z^2)/(a+bz)$. The function $\phi$ has a single pole in $\bT_+$ at $z_0=-a/b=(\sqrt{5}-3)/2$ with the residue $\gm_0=(25-11\sqrt{5})/2.$ Therefore, the function $\phi$ can be split into $\phi=\phi^++\phi^-$, where $\phi^-(z)=\gm_0/(z-z_0)$ and $\phi^+\in\calR_+$. Hence, we have to take the $2\tm 2$ matrix function
$$
F(z)=\begin{pmatrix}1&0\\
\frac{\gm_0}{z-z_0} &1\end{pmatrix}
$$
in the role of \eqref{2.1} and construct a paraunitary matrix $U$ according to Theorem \ref{Th1}. Postmultiplying $M$ by $U$, we get a spectral factor
\begin{equation}\label{SMU}
S_+(z)=M(z)U(z).
\end{equation}
Formulas \eqref{ui} and \eqref{um} suggest that we have to search solutions $u_1$ and $u_2$ of the system \eqref{7.1} in the form
$$
\wdt{u_1}(z)=C_1+\frac{C_{11}}{z-z_0}\;\;\text{ and }\;\; \wdt{u_2}(z)=C_2+\frac{C_{21}}{z-z_0}\,.
$$
Consequently,
$$
u_1(z)=\frac{C_1+(C_{11}-C_1z_0)z}{1-z_0z}\;\;\text{ and }\;\; u_2(z)=\frac{C_2+(C_{21}-C_2z_0)z}{1-z_0z}
$$
(since we deal with real coefficients, we do not use the conjugate sign for $C_i$) and the system derived from the equations \eqref{U1I}, \eqref{GmcC}, and \eqref{GmcC2} has the form
\begin{equation}\label{4sys}
  \begin{cases}
    C_1+\frac{C_{11}}{1-z_0}=1\\  C_2+\frac{C_{21}}{1-z_0}=0\\
    \gm_0\left(C_2+\frac{C_{21}z_0}{1-z_0^2}\right)-C_{11}=0\\
    \gm_0\left(C_1+\frac{C_{11}z_0}{1-z_0^2}\right)+C_{21}=0.
  \end{cases}
\end{equation}
For $u_{12}$ and $u_{22}$, we just have to change the right-hand side of \eqref{4sys} to $(0,1,0,0)^T$.

The code was designed to perform the exact arithmetics in the quadratic  field $\bQ(\sqrt{5})$ and the following solutions of \eqref{4sys} and its companion system were obtained by the Gaussian elimination:  $C_1=(35-7\sqrt{5})/20$, $C_{11}=(-11+5\sqrt{5})/4$, $C_2=(5-\sqrt{5})/20$,  $C_{21}=(-3+\sqrt{5})/4$ for $u_{11}$, $u_{21}$, and
         $C_1=(-5+\sqrt{5})/20$, $C_{11}=(3-\sqrt{5})/4$, $C_2=(35-7\sqrt{5})/20$,  $C_{21}=(-11+5\sqrt{5})/4$ for $u_{12}$, $u_{22}$. Hence the unitary matrix \eqref{2.5} was constructed
\begin{equation}\label{U22}
 \begin{pmatrix}
 \frac{35-7\sqrt{5}}{20}+\frac{-11+5\sqrt{5}}{4(1/z-z_0)}&
 \frac{-5+\sqrt{5}}{20}+\frac{3-\sqrt{5}}{4(1/z-z_0)}\\
 \frac{5-\sqrt{5}}{20}+\frac{-3+\sqrt{5}}{4(z-z_0)}&
 \frac{35-7\sqrt{5}}{20}+\frac{-11+5\sqrt{5}}{4(z-z_0)}
 \end{pmatrix} =
 c \begin{pmatrix}\frac{7+3z}{az+b}&\frac{-1+z}{az+b}\\[2mm]
 \frac{-1+z}{a+bz}& \frac{3+7z}{a+bz}
 \end{pmatrix},
\end{equation}
where $c=b(5-\sqrt{5})/20=\sqrt{3+\sqrt{5}}\sqrt{25-10\sqrt{5}+5}/20=2\sqrt{10}/20=1/\sqrt{10}$.

The spectral factor \eqref{SMU} is equal to the product of \eqref{Mz2} and \eqref{U22}. Given that the theory implies that the entries of  $S_+$ are polynomials of order $2$, we can assert that the obtained rational functions are indeed polynomials, allowing for exact divisions. Indeed we get
$$
S_+= c \begin{pmatrix}{7+3z}&-1+z\\
 {24+16z}& {-2+2z}
 \end{pmatrix},
$$
a result that can be verified through direct multiplication.

\subsection*{5.2}
In the next example, we construct a paraunitary matrix with the first row
\begin{equation}\label{91}
   V_1(z)=\begin{pmatrix}
         \frac{3z+3}{5z+6} & \frac{4z+5}{5z+6}& \frac{z+1}{6z+5}
       \end{pmatrix}     =:(v_1(z), v_2(z),\wdt{v_3}(z)).
\end{equation}
One can check that $|v_1(z)|^2+|v_2(z)|^2+|v_3(z)|^2=1$ for each $z\in\bT$, i.e., condition \eqref{112} of Theorem \ref{Th3} is satisfied and we follow the procedures described in its proof in order to perform this construction.

Since $v_3(z)=(z+1)/(5z+6)$ is analytic in $\bT_+$ together with its inverse, the solution $\phi_1, \phi_2$ of the system \eqref{7.12} can be identified by
$$
\phi_1=\left[\frac{\wdt{v_1}}{v_3}\right]^-= \frac{11}{2(6z+5)}=:\frac{\gm_1}{z-z_0}
 \text{ and } \phi_2=\left[\frac{\wdt{v_2}}{v_3}\right]^-= \frac{-11}{6(6z+5)}=:\frac{\gm_2}{z-z_0},
$$
where $\gm_1=11/12$, $\gm_2=-11/36$, and $z_0=-5/6$. Thus, we construct the matrix function $F$ of the form \eqref{2.1} and search for the corresponding paraunitary matrix \eqref{2.5}. Since $z_0=-5/6$ is a single pole of functions $\phi_1$ and $\phi_2$, we search the solutions $(u_1,u_2,u_3)$ of the system \eqref{7.1} in the form
$$
\wdt{u_j}(z)=C_j+\frac{C_{j1}}{z-z_0}\,,\;\;\;j=1,2,3,
$$
and the system derived from the equations \eqref{U1I}, \eqref{GmcC}, and \eqref{GmcC2} has the form
\begin{equation}\label{6sys}
  \begin{cases}
    C_1+\frac{C_{11}}{1-z_0}=\dt_{i1}\\
     C_2+\frac{C_{21}}{1-z_0}=\dt_{i2}\\
     C_3+\frac{C_{31}}{1-z_0}=\dt_{i3}\\
    \gm_1\left(C_3+\frac{C_{31}z_0}{1-z_0^2}\right)-C_{11}=0\\
    \gm_2\left(C_3+\frac{C_{31}z_0}{1-z_0^2}\right)-C_{21}=0\\
    \gm_1\left(C_1+\frac{C_{11}z_0}{1-z_0^2}\right)+\gm_2\left(C_2+\frac{C_{21}z_0}{1-z_0^2}\right)+C_{31}=0.
  \end{cases}
\end{equation}
Taking $i=1,2$, and $3$ in the system \eqref{6sys}, we get the coefficients of the entries of the first, second and third columns of \eqref{2.5}. Thus we obtain
$$
U(z)=\begin{pmatrix}
\frac{19}{22}+\frac{3z}{2(5z+6)} & \frac{1}{22}-\frac{z}{2(5z+6)} & \frac{1}{22}-\frac{z}{2(5z+6)}\\[1mm]
\frac{1}{22}-\frac{z}{3(5z+6)} & \frac{65}{66}+\frac{z}{6(5z+6)} & \frac{65}{66}+\frac{z}{6(5z+6)}\\[1mm]
\frac{1}{22}-\frac{1}{2(6z+5)} & \frac{-1}{66}+\frac{1}{6(6z+5)} & \frac{-1}{66}+\frac{1}{6(6z+5)}
\end{pmatrix}.
$$
For the function $V_1$ defined by \eqref{91}, we have $V_1(1)=\frac{1}{11}(6,9,2)$. We used the Cayley transform $V=(I-A)(I+A)^{-1}$ and searched for skew-symmetric matrix $A$ which satisfies the conditions $(I-A)_1=V_1(1)(I+A)$, where $(I-A)_1$ is the first row of $I-A$. This provides the orthogonal matrix $W$ with rational entries and the first column $V_1^T(1)$:
$$
W=\frac{1}{55}\begin{pmatrix}
                30 & -45 & -10 \\
                45 & 26 & 18 \\
                10 & 18 & -51
              \end{pmatrix}
$$
Using formula \eqref{UW}, we get a paraunitary matrix $V$ with the first column $V_1^T$:
$$
V(z)=\begin{pmatrix}
\frac{1}{2}+\frac{z}{2(5z+6)} & -\frac{7}{10}-\frac{13z}{10(5z+6)} & -\frac{1}{10}-\frac{9z}{10(5z+6)}\\[1mm]
\frac{5}{6}-\frac{z}{6(5z+6)} & \frac{13}{30}+\frac{13z}{30(5z+6)} & \frac{3}{10}+\frac{3z}{10(5z+6)}\\[1mm]
\frac{1}{6}+\frac{1}{6(6z+5)} & \frac{7}{30}+\frac{31}{30(6z+5)} & -\frac{4}{5}-\frac{7}{5(6z+5)}
\end{pmatrix}=\begin{pmatrix}
         \frac{3z+3}{5z+6} & -\frac{24z+21}{5(5z+6)}& -\frac{7z+3}{5(5z+6)}\\
         \frac{4z+5}{5z+6} & \frac{13z+13}{5(5z+6)} &\frac{9z+9}{5(5z+6)}\\
         \frac{z+1}{6z+5}  & \frac{7z+11}{5(6z+5)}& -\frac{24z+27}{5(6z+5)}
       \end{pmatrix}
$$
and the matrix we search for is $V^T$.

\section{Acknowledgments}
{
The work was supported by the EU through the H2020-MSCA-RISE-2020 project EffectFact, Grant agreement ID: 101008140. The third author was also supported by Faculty Research funding from the Division of Science and Mathematics, NYUAD.
}


\begin{thebibliography}{10}

\bibitem{AAM}
V.~M. Adukov, N.~V. Adukova,   G.~Mishuris,   \emph{An explicit Wiener–Hopf factorization algorithm for matrix polynomials
 and its exact realizations within ExactMPF package}, Proc. R. Soc. A. {478} (2022),   20210941.

\bibitem{AndesonMoore}
B.~D.~O. Anderson,  J.~B. Moore, \emph{Linear optimal control},
  Prentice-Hall, Inc., Englewood Cliffs, N.J., 1971.

\bibitem{BLM}
J.~R. Barry, E.~A. Lee,  D.~G. Messerschmitt, \emph{Digital communication},
  Springer, Berlin, 2004.

\bibitem{Bott13}
A.~B{\"o}ttcher, M.~Halwass, \emph{A {N}ewton method for canonical
  {W}iener-{H}opf and spectral factorization of matrix polynomials}, Electron.
  J. Linear Algebra {26} (2013), 873--897.

\bibitem{Dau}
I.~Daubechies, \emph{Ten lectures on wavelets}, SIAM, Philadelhia, PA, 1992.

\bibitem{Davis}
J.~H. Davis, \emph{Foundations of deterministic and stochastic control},
  Birkh\"{a}user Boston, Inc., Boston, MA, 2002.

\bibitem{EH}
M.~Ehler,  B.~Han, \emph{ Wavelet bi-frames with few generators from
multivariate refinable functions}, Appl. Comput. Harmon. Anal. {
25}, (2008), 407-414

\bibitem{Eph14}
L.~Ephremidze, \emph{An elementary proof of the polynomial matrix spectral
	factorization theorem}, Proc. Roy. Soc. Edinburgh Sect. A, {144},
	(2014), 747--751.
	
\bibitem{TR22}
L.~Ephremidze, A.~Gamkrelidze,  I.~Spitkovsky, \emph{On the spectral
  factorization of singular, noisy, and large matrices by
  {J}anashia-{L}agvilava method}, Trans. A. Razmadze Math. Inst. {176}
  (2022), no.~3, 361--366.

\bibitem{EJL11}
L.~Ephremidze, G.~Janashia,  E.~Lagvilava, \emph{On approximate spectral
  factorization of matrix functions}, J. Fourier Anal. Appl. {17}
  (2011), no.~5, 976--990.

\bibitem{EL2014}
L.~Ephremidze,  E.~Lagvilava, \emph{On compact wavelet matrices of rank {$m$}
  and of order and degree {$N$}}, J. Fourier Anal. Appl. {20} (2014),
  no.~2, 401--420. \MR{3200928}

\bibitem{IEEE2018}
L.~Ephremidze, F.~Saied,  I.~M. Spitkovsky, \emph{On the algorithmization of
  {J}anashia-{L}agvilava matrix spectral factorization method}, IEEE Trans.
  Inform. Theory {64} (2018), no.~2, 728--737.

\bibitem{IEC21}
L.~Ephremidze,  I.~Spitkovsky, \emph{An algorithm for $J$-spectral
  factorization of certain matrix functions}, 2021 60th IEEE Conference on
  Decision and Control (CDC), 2021, pp.~5820--5825.

\bibitem{JMAA22}
L.~Ephremidze,  I.~M. Spitkovsky, \emph{On multivariable matrix spectral
  factorization method}, J. Math. Anal. Appl. {514} (2022), no.~1, Paper
  No. 126300, 25.

\bibitem{Fischer}
R.~Fischer, \emph{Precoding and signal shaping for digital transmission}, New
  York: Wiley, 2002.

  \bibitem{HZ}
B.~Han,  X.~Zhuang, \emph{Algorithms for matrix extension and
orthogonal wavelet filter banks over algebraic number fields},
 Math. Comp.  {82}, (2013),  459--490 .

\bibitem{Jaf}
A.~Jafarian,  J.~G. McWhirter, \emph{A novel method for multichannel spectral
  factorization}, Proc. Europ. Signal Process. Conf. (2012), 1069-1073.

\bibitem{JL99}
G.~Janashia,  E.~Lagvilava, \emph{A method of approximate factorization of
  positive definite matrix functions}, Studia Math. {137} (1999), no.~1,
  93--100. \MR{1735630 (2000m:15015)}

\bibitem{IEEE2011}
G.~Janashia, E.~Lagvilava,  L.~Ephremidze, \emph{A new method of matrix
  spectral factorization}, IEEE Trans. Inform. Theory {57} (2011),
  no.~4, 2318--2326.

\bibitem{SA4}
V.~Kolev, T.~Cooklev,  F.~Keinert, \emph{{M}atrix spectral factorization for
  {SA}4 multiwavelet}, Multidimens. Syst. Signal Process. {29} (2018),
  1613--1641.

\bibitem{Kolm41}
A.~N. Kolmogoroff, \emph{Stationary sequences in {H}ilbert's space}, Bolletin
  Moskovskogo Gosudarstvenogo Universiteta. Matematika {2} (1941).

\bibitem{Kuc}
V.~Ku{\v{c}}era, \emph{Factorization of rational spectral matrices: A survey of
  methods}, in Proc. IEEE Int. Conf. Control, Edinburgh {2} (1991),
  1074--1078.

   \bibitem{LLS}
W.~Lawton, S.~L.~Lee,  Z.~Shen, \emph{An algorithm for matrix extension
and wavelet construction},  Math. Comp.  {65}, (1996), 723--737.

\bibitem{CNR}
J.~N. MacLaurin,  P.~A. Robinson, \emph{Determination of effective brain
  connectivity from activity correlations}, Phys. Rev. E {99} (2019),
  042404.


  \bibitem{Park}
P.~Park, \emph{A Computational Theory of Laurent Polynomial Rings and
Multidimensional FIR Systems}. PhD thesis, UC Berkeley, CA (1995)

  \bibitem{RW}
H.~R. Resnikoff,  R.~O. Wells,  \emph{ Wavelet Analysis}. Springer-Verlag,
1998.


  \bibitem{ACHA21}
C.~Ri,  Y. Paek, \emph{Causal FIR symmetric paraunitary matrix extension and
construction of symmetric tight $M$-dilated framelets},  Appl. Comput. Harmon. Anal. {
51}, (2021), 437-460.

\bibitem{SayKai}
A.~H. Sayed,  T.~Kailath, \emph{A survey of spectral factorization methods},
  Numer. Linear Algebra Appl. {8} (2001), no.~6-7, 467--496.


  \bibitem{Vai}
P.~P. Vaidyanathan,  \emph{ Multirate Systems and Filter Banks}, Prentice
Hall, New Jersey, 1993.

\bibitem{Wie58}
N.~Wiener,  P.~Masani, \emph{The prediction theory of multivariate stochastic
  processes. {II}. {T}he linear predictor}, Acta Math. {99} (1958),
  93--137.

\end{thebibliography}

\def\cprime{$'$}
\providecommand{\bysame}{\leavevmode\hbox to3em{\hrulefill}\thinspace}
\providecommand{\MR}{\relax\ifhmode\unskip\space\fi MR }
\providecommand{\MRhref}[2]{%
  \href{http://www.ams.org/mathscinet-getitem?mr=#1}{#2}
}
\providecommand{\href}[2]{#2}

\end{document}